\documentclass[12pt]{article}
 \usepackage{latexsym}
 \usepackage{amssymb}
 \usepackage{amsthm}
 \usepackage{color}
 \usepackage{graphicx}
\usepackage{epstopdf}
 \newtheorem{Theorem}{Theorem}
\newtheorem{Definition}{Definition}
\newtheorem{Proposition}{Proposition}

\newtheorem{Lemma}{Lemma}

\newcommand{\A}{{\cal A}}

\newcommand{\HH}{{\cal H}}

\newcommand{\uu}{{\bf u}}
\newcommand{\vv}{{\bf v}}
\newcommand{\x}{{\bf x}}
\newcommand{\y}{{\bf y}}

\newcommand{\e}{{\bf e}}
\newcommand{\0}{{\bf 0}}

\begin{document}
\title{New Classes of Positive Semi-Definite Hankel Tensors}

\author{Qun Wang \footnote{Department of Applied
 Mathematics, The Hong Kong Polytechnic University, Hung Hom, Kowloon, Hong
 Kong.   Email: wangqun876@gmail.com (Q. Wang).}
\quad Guoyin Li
\footnote{Department of Applied Mathematics, University of New South
Wales, Sydney 2052, Australia. E-mail: g.li@unsw.edu.au (G. Li).
This author's work was partially supported by Australian Research
Council.}
\quad Liqun Qi\footnote{Department of Applied
 Mathematics, The Hong Kong Polytechnic University, Hung Hom, Kowloon, Hong
 Kong.
 E-mail: maqilq@polyu.edu.hk (L. Qi).  This author's work was partially supported by the Hong Kong Research Grant Council (Grant No. PolyU
501212, 501913, 15302114 and 15300715).}
\quad Yi Xu \footnote{Department of Mathematics, Southeast
University, Nanjing, 210096 China.   Email: yi.xu1983@gmail.com (Y. Xu).}}
\date{\today} \maketitle

\begin{abstract}
\noindent  

A Hankel tensor is called a strong Hankel tensor if the Hankel matrix generated by its generating vector is positive semi-definite.  It is known that an even order strong Hankel tensor is a sum-of-squares tensor, and thus a positive semi-definite tensor.   The SOS decomposition of strong Hankel tensors has been well-studied by Ding, Qi and Wei \cite{DQW1}.  On the other hand, very little is known for positive semi-definite Hankel tensors which are not strong Hankel tensors.   In this paper, we study some classes of positive semi-definite Hankel tensors which are not strong Hankel tensors.  These include truncated Hankel tensors and quasi-truncated Hankel tensors.   Then we show that a strong Hankel tensor generated by an absoluate integrable function is always completely decomposable, and give a class of SOS Hankel tensors which are not completely decomposable.

\noindent {\bf Key words:}\hspace{2mm} Hankel tensors, generating vectors, positive semi-definiteness,
strong Hankel tensors.

\noindent {\bf AMS subject classifications (2010):}\hspace{2mm}
15A18; 15A69
  \vspace{3mm}

\end{abstract}

\section{Introduction}
\hspace{4mm}

Hankel tensors have important applications in signal processing \cite{BB,BDA,CQW,DQW}, automatic control \cite{Sm}, and geophysics \cite{OS,TBM}.
For example, the positive semi-definiteness of Hankel tensor can be a condition for the multidimensional moment problem is solvable or not \cite{Be,Las,Qi16}.


%
In \cite{Qi15}, two classes of positive semi-definite (PSD) Hankel tensors were identified: even order strong Hankel tensors and even order complete Hankel tensors.
It was proved in \cite{LQX} that complete Hankel tensors are strong Hankel tensors, and even order strong Hankel tensors are SOS (sum-of-squares) tensors.
In \cite{LQW16}, generalized anti-circulant tensors were studied, which are one special class of Hankel tensors.
The necessary and sufficient conditions for positive semi-definiteness of even order generalized anti-circulant tensors in some cases were given, and the tensors are strong Hankel tensors and SOS tensors in these cases.
Inheritance property was given in \cite{DQW1}, which means that if a lower-order Hankel tensor is positive semi-definite (or positive definite, or negative semi-definite, or negative definite, or SOS), then its associated higher-order Hankel tensor with the same generating vector, where the higher order is a multiple of the lower order, is also positive semi-definite (or positive definite, or negative semi-definite, or negative definite, or SOS, respectively). The SOS decomposition of strong Hankel tensors have also been given in \cite{DQW1}.
The inheritance property established in \cite{Qi15} about strong Hankel tensor can be regarded as a special case of this inheritance property.
There are other results about PSD Hankel tensors, SOS Hankel tensors and PNS non-SOS (short for PNS as in \cite{Ch}) Hankel tensors and some regions which do not exist PNS Hankel tensors were given \cite{CQW16}. 
A recent detailed study on general SOS tensors can be found in \cite{CQL}. 

Denote that $[n] := \{1, \cdots, n\}$. Tensor $\A$ is said to be a symmetric tensor if its entries $a_{i_1\cdots i_m}$ is invariant under any index permutation. Denote the set of all the real symmetric tensors of order $m$ and dimension $n$ by $S_{m,n}$.

Let $\vv = (v_0, \cdots, v_{(n-1)m})^\top$. Define $\A = (a_{i_1\cdots i_m})\in S_{m,n}$ by
\begin{equation}\label{e1-1}
  a_{i_1\cdots i_m} = v_{i_1+\cdots+i_m-m},
\end{equation}
for $i_1, \cdots, i_m \in [n]$. Then $\A$ is a {\bf Hankel tensor} \cite{LQX,PDV,Qi15} and $\vv$ is called the {\bf generating vector} of $\A$.
An order $m$ dimensional $n$ Hilbert tensor $\HH$ is a Hankel tensor with $\vv = (1, \frac{1}{2}, \frac{1}{3},\cdots, \frac{1}{nm})^\top$, and even order Hilbert tensors are
positive definite \cite{SQ}.

Let $\x \in \Re^n$. Then $\x^m$ is a rank-one symmetric tensor with entries $x_{i_1}\cdots x
_{i_m}$.  For $\A \in S_{m,n}$ and $\x \in \Re^n$, we have a homogeneous polynomial $f(\x)$ of $n$ variables and degree $m$,
\begin{equation}\label{e1-2}
  f(\x) = \A\x^{\otimes m} \equiv \sum_{i_1,\cdots,i_m \in [n]} a_{i_1\cdots i_m} x_{i_1}\cdots x_{i_m}.
\end{equation}
Note that there is a one to one relation between homogeneous polynomials and symmetric
tensors.
If $f(\x)\geq 0$ for all $\x \in \Re^n$, then homogeneous polynomial $f(\x)$ and symmetric tensor $\A$ are called {\bf positive semi-definite}(PSD). If $f(\x) > 0$ for all $\x \in \Re^n$, $\x \neq 0$, then $f(x)$ and $\A$ are called {\bf positive definite} (PD). The concepts of positive semi-definite and positive definite symmetric tensors were introduced in \cite{Qi}. The problem for determining a given even order symmetric tensor is positive semi-definite or not has important applications in engineering and science \cite{DQW,HLQS,Re07}. If $\A$ is a Hankel tensor, then $f(\x)$ is called a Hankel polynomial.

Let $A = (a_{ij})$ be an $\lceil\frac{(n-1)m+2}{2}\rceil \times \lceil\frac{(n-1)m+2}{2}\rceil$ matrix with $a_{ij} \equiv v_{i+j-2}$, where $v_{2\lceil\frac{(n-1)m}{2}\rceil}$ is an additional number when $(n-1)m$ is odd. Then $A$ is a Hankel matrix, associated with the Hankel tensor $\A$. When $(n-1)m$ is even, such an associated Hankel matrix is unique. Recall from \cite{Qi15} that $\A$ is called a {\bf strong Hankel tensor} if there exists an associated Hankel matrix $A$ which is positive semi-definite.

Let $g(\y) = \y^\top A\y$, where $\y = (y_1, \cdots, y_{\frac{(n-1)m+2}{2}})^\top$ and
$A$ is an associated Hankel matrix of $\A$. Then, $\A$ is a strong Hankel tensor if and only if $g$ is PSD for at least one associated Hankel matrix $A$ of $\A$.

Another class of Hankel tensors given by Qi in \cite{Qi15} are complete Hankel tensors.
A vector $\uu = (1, \gamma, \gamma^2, \cdots , \gamma^{n-1})^\top$ for some $\gamma \in \Re$ is called a Vandermonde vector \cite{Qi15}.
If tensor $\A$ has the form
\begin{eqnarray*}
  \A = \sum_{i \in [r]}\alpha_i(\uu_i)^{\otimes m},
\end{eqnarray*}
where $\uu_i$ for $i = 1, \cdots, r$, are all Vandermonde vectors, then we say that $\A$ has a Vandermonde decomposition. It was shown in \cite{Qi15} that a symmetric tensor is a Hankel tensor if and only if it has a Vandermonde decomposition. If the coefficients $\alpha_i$ for $i = 1, \cdots, r$, are
all nonnegative, then $\A$ is called a complete Hankel tensor \cite{Qi15}.
It was proved in \cite{Qi15} that even order strong or complete Hankel tensors are positive semi-definite.

Let $m = 2k$. If $f(\x)$ can be decomposed to the sum of squares of polynomials of degree $k$, then $f(\x)$ is called a \textbf{sum-of-squares (SOS) polynomial}, which means there exist forms $g_1(\x),\cdots, g_k(\x)$ of degree $k$ such that
\begin{equation}\label{e-sos}
  f(\x)=\sum\limits_{i=1}^kg_i(\x)^2,
\end{equation}
and the corresponding symmetric tensor $\A$ is called an \textbf{SOS tensor} \cite{HLQ} (for a recent study see also \cite{CQL}).

Clearly, an SOS polynomial (tensor) is a PSD polynomial (tensor), but not vice versa. In 1888, young Hilbert \cite{Hi} proved that for homogeneous polynomials, only in the following three cases, a positive semi-definite form definitely is a sum-of-squares polynomial: 1) $m = 2$; 2) $n = 2$; 3) $m=4$ and $n=3$, where $m$ is the degree of the polynomial and $n$ is the number of variables.
Hilbert proved that in all the other possible combinations of $n$ and even $m$, there are PNS homogeneous polynomials.
The first PNS homogeneous polynomial is the Motzkin function \cite{Mo} with $m=6$ and $n=3$.
Other examples of PNS homogeneous polynomials was found in \cite{AP, Ch, CL, Re}.

 Let $\A \in  S_{m, n}$.  If there are positive integer $r \in \mathbb{N}$ and vectors $\x_j \in \Re^n$ for $j \in [r]$ such that
\begin{equation}\label{eq:lala}
\A = \sum_{j \in [r]} \x_j^{\otimes m},\end{equation}
then we say that $\A$ is a {\bf completely decomposable tensor}.   If $\A$ admits a decomposition (\ref{eq:lala}) with $\x_j \in \Re^n_+$ for all $j \in [r]$, then $\A$ is called a {\bf completely positive tensor} \cite{QXX}.
Clearly, a complete Hankel tensor is a completely decomposable tensor.

By \cite{Qi15}, a necessary condition for $\A$ to be PSD is that
\begin{equation} \label{e1-3}
  v_{(i-1)m} \geq 0 \quad \mbox{for} \quad i \in [n].
\end{equation}

We know many properties of strong Hankel tensors. However, we know little about PSD Hankel tensors but not strong Hankel tensors.
In this paper, we present some classes of Hankel tensors which are PSD but not strong Hankel tensors, including truncated Hankel tensors and  quasi-truncated Hankel tensors.
Then we show that strong Hankel tensors are always completely decomposable, and give a class of SOS Hankel tensors which are not completely decomposable.

The remainder of this paper is organized as follows.

In the next section, we introduce truncated Hankel tensors which are of odd dimension, i.e., $n$ is odd, and whose generating vector $v$ has only three nonzero entries: $v_0$, $v_{(n-1)m \over 2}$ and $v_{(n-1)m}$.    Since we are only concerned about PSD Hankel tensors, by (\ref{e1-3}), we assume that these three entries are all nonnegative.     Under this assumption, we show that a truncated Hankel tensor $\A$ is not a strong Hankel tensor as long as $v_{(n-1)m \over 2}$ is positive.   Then we show that when $m$ is even, $n=3$ and $v_0 = v_{2m}$, 
there are two numbers $d_1=d_1(m)$ and $d_2=d_2(m)$ with $0<d_2\leq d_1$, such that if $v_0\geq d_1v_m$, $\A$ is an SOS tensor; and if $v_0\geq d_2$, $\A$ is an PSD tensor.
Then, for $m=6$ and $n=3$, we show that for a truncated Hankel tensor $\A$, the following three statements are equivalent: 1) $\A$ is PSD; 2) $\A$ is SOS; 3) $v_0v_{12} \ge v_6^2d$; and give an explicit value of $d$.

In Section 3, we introduce quasi-truncated Hankel tensors which are of odd dimension, i.e., $n$ is odd, and whose generating vector $v$ has only five nonzero entries: $v_0$, $v_1$, $v_{(n-1)m \over 2}$, $v_{(n-1)m-1}$ and $v_{(n-1)m}$.    Again, since we are only concerned about PSD Hankel tensors, by (\ref{e1-3}), we assume that $v_0$, $v_{(n-1)m \over 2}$ and $v_{(n-1)m}$ are nonnegative.   Under this assumption, we show that a quasi-truncated Hankel tensor is not a strong Hankel tensor as long as $v_{(n-1)m \over 2}$ is positive.   Then we give a necessary condition for a sixth order quasi-truncated Hankel tensor $\A$ to be PSD, a sufficient condition for $\A$ to be SOS, respectively.

In Section 4, we define completely decomposable tensors, show that strong Hankel tensors  generated by absoluate integrable functions are always completely decomposable, and give a class of SOS Hankel tensors which are not completely decomposable.

\section{Truncated Hankel Tensors}
\hspace{4mm}

In this section, we consider the cases that the Hankel tensors $\A$ are generated by $\vv = (v_0, 0, \cdots, 0, v_{\frac{(n-1)m}{2}}, 0, \cdots,  0, v_{(n-1)m})^\top$ where $n$ is odd. 
We call such Hankel tensors {\bf truncated Hankel tensors}.

If $\vv = (v_0, 0, \cdots, 0, v_{\frac{(n-1)m}{2}}, 0, \cdots,  0, v_{(n-1)m})^\top$ where $n$ is odd, (\ref{e1-2}) and $g(\y)$ have the simple form
\begin{eqnarray} \label{e2-1}
f(\x)&=& v_0x_1^m + v_{(n-1)m}x_n^m\nonumber\\
&&+ v_{\frac{(n-1)m}{2}}\sum \left\{\textstyle {m\choose t_1} {m-t_1\choose t_2}\cdots {m-t_1-t_2-\cdots-t_{n-2}\choose t_{n-1}} x_1^{t_1}x_2^{t_2}\cdots x_n^{m-t_1-t_2-\cdots-t_{n-1}}\right. \nonumber\\
&& \quad \quad\quad \quad \quad \quad \left. : (n-1)t_1+(n-2)t_2+\cdots +t_{n-1}=\textstyle\frac{(n-1)m}{2}\right\}.\nonumber\\
\end{eqnarray}
and
\begin{equation} \label{e2-2}
g(\y) = v_0y_1^2+ v_{(n-1)m}y_{\frac{(n-1)m+2}{2}}^2 + v_{\frac{(n-1)m}{2}}\left(y^2_{\frac{(n-1)m}{4}+1}+\sum_{i \neq j}\left\{y_i y_j:i+j=\textstyle\frac{(n-1)m}{2}+2\right\}\right).
\end{equation}

Since we are only concerned about PSD Hankel tensors, we may assume that (\ref{e1-3}) holds.  From (\ref{e2-1}) and (\ref{e2-2}), we have the following proposition.

\begin{Proposition} \label{p2-1}
Suppose that (\ref{e1-3}) holds. If $v_{\frac{(n-1)m}{2}} = 0$, then the truncated Hankel tensor $\A$ is a strong Hankel tensor, and furthermore an SOS Hankel tensor if $m$ is even.   If $v_{\frac{(n-1)m}{2}}  > 0$, then $\A$ is not a strong Hankel tensor.
\end{Proposition}

\begin{proof}  When $v_{\frac{(n-1)m}{2}} =0$, from (\ref{e2-1}) and (\ref{e2-2}), we see that the truncated Hankel tensor $\A$ is a strong Hankel tensor,  and furthermore an SOS Hankel tensor if $m$ is even.   If $v_{\frac{(n-1)m}{2}} > 0$, consider $\bar \y = \e_i-\e_j$ where $i+j=\textstyle\frac{(n-1)m}{2}+2$ ,$i \neq j$ and $i\neq 1$ or $\frac{(n-1)m+2}{2}$.  We see that $g(\bar \y) = -2 v_{\frac{(n-1)m}{2}} < 0$.   Hence $\A$ is not a strong Hankel tensor in this case.
\end{proof}

We consider the case that $m \ge 6$ with $m$ is even and $n=3$. Assume that $v_0 = v_{2m}$.   We have the following theorem.

\begin{Theorem} \label{t2-1}
  Suppose that $\A=(a_{i_1\cdots i_m})$ where $m \ge 6$, $m$ is even and $n=3$ is a truncated Hankel tensor and (\ref{e1-3}) holds. Assume that $v_0 = v_{2m}$.   Let $\vv = (v_0, 0, \cdots, 0, v_m, 0, \cdots,  0, v_0)^\top$£¬ and
    \begin{eqnarray*}
d_1(m) & =& \inf\{d>0: \A \mbox{ is an SOS tensor} \mbox{ for all } \vv \mbox{ such that } v_0\geq d \, v_m\},   \\
d_2(m) & =& \inf\{d>0: \A \mbox{ is a PSD tensor} \mbox{ for all } \vv \mbox{ such that } v_0\geq d \, v_m\}.
  \end{eqnarray*}
  Then
  (1) $d_2(m)=d_1(m)=0$ if $v_m=0$;  (2) $0 \le d_2(m) \le d_1(m)<+\infty$ if $v_m >0$.
\end{Theorem}
\begin{proof}
If $v_m=0$, it is clear that $f(\x)=v_0x_1^m + v_0x_3^m$ is  SOS (and in particular PSD) because $v_0,v_{2m} \ge 0$ and $m$ is even. Thus, $d_1(m)=d_2(m)=0$ in this case.
Now, let us consider the case where $v_m>0$.
  We rewrite (\ref{e2-1}) as
$$f(\x) = f_1(\x) + f_2(\x) + f_3(\x),$$
where
$$f_1(\x)={v_m \over 2}\sum_{p=1}^{m \over 2} \textstyle {m\choose p} \textstyle {m-p\choose m-2p}x_2^{m-2p}\left(x_1^p + x_3^p\right)^2,$$
$$f_2(\x) = v_0x_1^m + {v_m \over 2}x_2^m - {v_m \over 2}\sum_{p=1}^{m \over 2} \textstyle {m\choose p} \textstyle {m-p\choose m-2p}x_2^{m-2p}x_1^{2p}$$
and
$$f_3(\x) = v_0x_3^m + {v_m \over 2}x_2^m - {v_m \over 2}\sum_{p=1}^{m \over 2} \textstyle {m\choose p} \textstyle {m-p\choose m-2p}x_2^{m-2p}x_3^{2p}.$$
Clearly, $f_1(\x)$ is PSD and SOS. We now consider the terms in $f_2(\x)$.
For each $p = 1, \ldots, {m \over 2}$, choose a positive constant
   $\delta(p)$ such that
   \[
   1-\sum_{p=1}^{m \over 2} \frac{m-2p}{m} \delta(p) > 0.
   \]
   For each $p = 1, \ldots, {m \over 2}$, let $\Delta(p)$ be another positive constant such that
   \[
    \Delta(p)^{\frac{2p}{m}}\delta(p)^{\frac{m-2p}{m}}=\textstyle {m\choose p} \textstyle {m-p\choose m-2p}.
   \]
   Then, by the arithmetic-geometric inequality, for each $p = 1, \ldots, {m \over 2}$,
   \[
    \textstyle {m\choose p} \textstyle {m-p\choose m-2p}x_2^{m-2p}x_1^{2p} = \bigg(  \delta(p) x_2^m \bigg)^{\frac{m-2p}{m}}\bigg( \Delta(p) x_1^m \bigg)^{\frac{2p}{m}}
    \le \frac{m-2p}{m} \bigg(  \delta(p) x_2^m \bigg) +{\frac{2p}{m}}\bigg( \Delta(p)x_1^m \bigg).
   \]
This shows that, for each $p = 1, \ldots, {m \over 2}$,
\[
\frac{m-2p}{m} \bigg(  \delta(p) x_2^m \bigg) +{\frac{2p}{m}}\bigg( \Delta(p)x_1^m \bigg)-{m\choose p} \textstyle {m-p\choose m-2p}x_2^{m-2p}x_1^{2p} 
\]
is a PSD diagonal minus tail form, and hence SOS \cite{FK}. 
Note that $f_2$ can be written as 
\begin{eqnarray*}
f_2(\x) &=&  {v_{m} \over 2}\bigg(1-\sum_{p=1}^{m \over 2} \frac{m-2p}{m} \delta(p)\bigg) x_2^m  \\
& & + {v_m \over 2} \, \sum_{p=1}^{m \over 2}  \bigg(\frac{m-2p}{m} \delta(p) x_2^m  -
\textstyle {m\choose p} \textstyle {m-p\choose m-2p}x_2^{m-2p}x_1^{2p} +{\frac{2p}{m}} \, \Delta(p) x_1^m \bigg) \\
& & + v_0x_1^m
- {v_m \over 2} \sum_{p=1}^{m \over 2}{\frac{2p}{m}} \Delta(p) x_1^m.
\end{eqnarray*}
Therefore, if
\[
v_0 \ge  \sum_{p=1}^{m \over 2}{\frac{p}{m}} \Delta(p) v_m,
\]
then $f_2$ is PSD and SOS.   Similarly, we may show that under the same condition, $f_3$ is also PSD and SOS. These, in particular, show that
\[
0 \le d_2(m) \le d_1(m) \le  \sum_{p=1}^{m \over 2}{\frac{p}{m}} \Delta(p)<+\infty.
\]

\end{proof}

We consider the simple case that the tensors $\A$ are sixth order three dimensional truncated Hankel tensors. Here we allow $v_0 \not = v_{12}$.  We give a necessary and sufficient condition that the sixth order three dimensional truncated Hankel tensors to be PSD, and show that such tensors are PSD if and only if they are SOS.

 The sixth order three dimensional truncated Hankel tensor $\A$ is generated by $\vv = (v_0, 0, 0, 0, 0, 0, v_6, 0, 0, 0, 0, 0, v_{12})^\top$.
Now, (\ref{e2-1}) and (\ref{e2-2}) have the simple form
\begin{equation} \label{e2-6}
f(\x) = v_0x_1^6 + v_6(x_2^6 +30x_1x_2^4x_3 +90x_1^2x_2^2x_3^2 + 20x_1^3x_3^3) + v_{12}x_3^6
\end{equation}
and
\begin{equation} \label{e2-7}
g(\y) = v_0y_1^2 + v_6(y_4^2 +2y_1y_7 +2y_2y_6 + 2y_3y_5) + v_{12}y_7^2.
\end{equation}

\begin{Theorem} \label{t2-2}
Suppose that $\A$ is a sixth order three dimensional truncated Hankel tensor, the following statements are equivalent:
\begin{itemize}
 \item[{\rm (i)}] The truncated Hankel tensor $\A$ is a PSD Hankel tensor;
 \item[{\rm (ii)}] The truncated Hankel tensor $\A$ is an SOS Hankel tensor;
 \item[{\rm (iii)}] The relation (\ref{e1-3}) holds and \begin{equation} \label{e2-8}
\sqrt{v_0v_{12}} \ge (560 +70\sqrt{70})v_6.
\end{equation}
\end{itemize}
Furthermore, the truncated Hankel tensor $\A$ is positive definite if and only if $v_0,v_6,v_{12} > 0$ and strict inequality holds in (\ref{e2-8}).
\end{Theorem}
\begin{proof}
$[{\rm (i)} \Rightarrow {\rm (iii)}]$ Suppose that $\A$ is PSD, then clearly $v_0, v_6, v_{12} \ge 0$. To see {\rm (iii)}, we only need to show  (\ref{e2-8}) holds.  Let $t \ge 0$ and let $\bar \x = (\bar x_1, \bar x_2, \bar x_3)^\top$, where
$$\bar x_1 = v_{12}^{1 \over 6},\  \bar x_2 = \sqrt{t}(v_0v_{12})^{1 \over 12}, \
\bar x_3 = - v_0^{1 \over 6}.$$
Substitute them to (\ref{e2-6}).  If $\A$ is PSD, then $f(\bar \x) \ge 0$.   It follows from (\ref{e2-6}) that
$$v_0v_{12} + v_6(t^3 -30t^2 + 90t - 20)\sqrt{v_0v_{12}} + v_0v_{12} \ge 0.$$
From this, we have
$$\sqrt{v_0v_{12}} \ge {-t^3 +30t^2 -90t +20 \over 2}v_6.$$
Substituting $t = 10 + \sqrt{70}$ to it, we have (\ref{e2-8}).

$[{\rm (iii)} \Rightarrow {\rm (ii)}]$ We now assume that (\ref{e1-3}) and (\ref{e2-8}) hold.   We will show that $\A$ is SOS.
If $v_6 = 0$, then by Proposition \ref{p2-1}, $\A$ is an SOS Hankel tensor.  Assume that $v_6 > 0$.   By (\ref{e2-8}), $v_0 > 0$ and $v_{12} > 0$.  We now have
$$f(\x) = 10v_6 \left(\left({v_0 \over v_{12}}\right)^{1 \over 4}x_1^3 + \left({v_{12} \over v_0}\right)^{1 \over 4}x_3^3\right)^2 + v_6\left( \sqrt{10-\sqrt{70} \over 2}x_2^3 + \sqrt{150+15\sqrt{70}}x_1x_2x_3\right)^2 +f_1(\x),$$
where
\begin{equation} \label{e2-9}
f_1(\x) = \left(v_0 -10v_6\left({v_0 \over v_{12}}\right)^{1 \over 2}\right)x_1^6 + {\sqrt{70}-8 \over 2}v_6x_2^6 + \left(v_{12} -10v_6\left({v_{12} \over v_0}\right)^{1 \over 2}\right)x_3^6 - (60+15\sqrt{70})v_6x_1^2x_2^2x_3^2.
\end{equation}
We see that $f_1(\x)$ is a diagonal minus tail form \cite{FK}.   By the arithmetic-geometric inequality, we have
$$\left(v_0 -10v_6\left({v_0 \over v_{12}}\right)^{1 \over 2}\right)x_1^6 + {\sqrt{70}-8 \over 2}v_6x_2^6 + \left(v_{12} -10v_6\left({v_{12} \over v_0}\right)^{1 \over 2}\right)x_3^6$$
$$\ge 3\left({\sqrt{70}-8 \over 2}v_6(\sqrt{v_0v_{12}} - 10v_6)^2\right)^{1 \over 3}x_1^2x_2^2x_3^2.$$
By (\ref{e2-8}),
\begin{equation} \label{e2-10}
 3\left({\sqrt{70}-8 \over 2} v_6 (\sqrt{v_0v_{12}} - 10v_6)^2\right)^{1 \over 3}x_1^2x_2^2x_3^2 \ge (60+15\sqrt{70})v_6x_1^2x_2^2x_3^2.
 \end{equation}
Thus, $f_1$ is a PSD diagonal minus tail form.   By \cite{FK}, $f_1$ is an SOS polynomial.   Hence, $f$ is also an SOS polynomial if (\ref{e1-3}) and (\ref{e2-8}) hold.

$[{\rm (ii)} \Rightarrow {\rm (i)}]$ This implication is direct by the definition.

We now prove the last conclusion of this theorem.   First, we assume that $\A$ is positive definite.  Then, $v_6 = f(\e_2) > 0$ as $\e_2 \not = \0$. Similarly, $v_0=f(\e_1)>0$ and $v_{12}=f(\e_3)>0$.  Note that in the above $[{\rm (i)} \Rightarrow {\rm (iii)}]$ part, $f(\bar \x)>0$ as $\bar \x \not = \0$.   Then strict inequality holds for the last two inequalities in
the above $[{\rm (i)} \Rightarrow {\rm (iii)}]$ part.  This implies that
strict inequality holds in (\ref{e2-8}).

On the other hand, assume that $v_0,v_6,v_{12} > 0$ and strict inequality holds in (\ref{e2-8}).  
 Let $\x = (x_1, x_2, x_3)^\top \not = \0$.   If $x_1 \not = 0, x_2 \not = 0$ and $x_3 \not = 0$, then strict inequality holds in (\ref{e2-10}) as $v_6 > 0$ and strict inequality holds in (\ref{e2-8}).   Then $f_1(\x ) > 0$.   If $x_2 \not = 0$ but $x_1x_3 = 0$, then
from (\ref{e2-9}), we still have $f_1(\x) > 0$.   If $x_2 = 0$ and one of $x_1$ and $x_3$ are nonzero, then we still have $f_1(\x) > 0$ by (\ref{e2-9}).   Thus, we always have $f_1(\x) > 0$ as long as $\x \not = \0$.   This implies $f(\x) > 0$ as long as $\x \not = \0$.   Hence, $\A$ is positive definite.
\end{proof}

\section{Quasi-Truncated Hankel Tensors}
\hspace{4mm}
In this section, we consider the case that the Hankel tensor $\A$ is generated by $\vv = (v_0, v_1, 0, \cdots, 0, v_{\frac{(n-1)m}{2}} , 0, \cdots, 0, v_{(n-1)m-1}, v_{(n-1)m})^\top$ where $n$ is odd.   Adding $v_1$ and $v_{(n-1)m-1}$ to the case in the last section, we get this case.   We call
such a Hankel tensor a {\bf quasi-truncated Hankel tensor}.   Hence, truncated Hankel tensors are quasi-truncated Hankel tensors.

Since we are only concerned about PSD Hankel tensors, we may assume that (\ref{e1-3}) holds.
Now, (\ref{e1-2}) and $g(\y)$ have the simple form
\begin{eqnarray} \label{e3-1}
f(\x)&=& v_0x_1^m+m v_1x_1^{m-1}x_2+m v_{(n-1)m-1}x_{n-1}x_n^{m-1}+ v_{(n-1)m}x_n^m\nonumber\\
&&+ v_{\frac{(n-1)m}{2}}\sum \left\{\textstyle {m\choose t_1} {m-t_1\choose t_2}\cdots {m-t_1-t_2-\cdots-t_{n-2}\choose t_{n-1}} x_1^{t_1}x_2^{t_2}\cdots x_n^{m-t_1-t_2-\cdots-t_{n-1}}\right. \nonumber\\
&&\quad \quad\quad \quad \quad \quad \left. : (n-1)t_1+(n-2)t_2+\cdots +t_{n-1}=\textstyle\frac{(n-1)m}{2}\right\}.\nonumber\\
\end{eqnarray}
and
\begin{eqnarray} \label{e3-2}
g(\y) &=& v_0y_1^2+2v_1y_1y_2+2v_{(n-1)m-1}y_{\frac{(n-1)m+2}{2}}y_{\frac{(n-1)m}{2}} +v_{(n-1)m}y_{\frac{(n-1)m+2}{2}}^2\nonumber\\
&&+ v_{\frac{(n-1)m}{2}}\left(y^2_{\frac{(n-1)m}{4}+1}+\sum_{i \neq j}\left\{y_i y_j:i+j=\textstyle\frac{(n-1)m}{2}+2\right\}\right).
\end{eqnarray}

We first show that a result that Proposition \ref{p2-1} continues to hold in this case.

\begin{Proposition} \label{p3-1}
Suppose that (\ref{e1-3}) holds and $m$ is even. If $v_{\frac{(n-1)m}{2}} = 0$, then the quasi-truncated Hankel tensor $\A$ is PSD if and only if $v_1 = v_{(n-1)m-1} = 0$.   In this case, $\A$ is a strong Hankel tensor and an SOS Hankel tensor.   If $v_{\frac{(n-1)m}{2}} > 0$, then $\A$ is not a strong Hankel tensor.
\end{Proposition}

\begin{proof}
Suppose that $v_{\frac{(n-1)m}{2}} =0$.   Assume that $v_1 \not = 0$.  If $v_0 = 0$, consider $\hat \x = (1, -v_1, 0,\cdots,0)^\top$.  Then $f(\hat \x)=-mv_1^2 < 0$.   If $v_0 > 0$, consider $\tilde \x = (1, -{v_0 \over v_1}, 0,\cdots,0)^\top$.  Then $f(\tilde \x) = (1-m)v_0 < 0$.  Thus, $\A$ is not PSD in these two cases.  Similar discussion holds for the case that $v_{(n-1)m-1} = 0$.   Assume now that $v_1 = v_{(n-1)m-1} = 0$.  By Proposition \ref{p2-1}, we see that the quasi-truncated Hankel tensor $\A$ is a strong Hankel tensor and an SOS Hankel tensor in this case.    This proves the first part of this proposition.

 Suppose that $v_{\frac{(n-1)m}{2}} > 0$.   Consider $\bar \y = \e_i-\e_j$ where $i+j=\textstyle\frac{(n-1)m}{2}+2$ ,$i \neq j$ and $i\neq 1$ or $\frac{(n-1)m+2}{2}$.  We see that $g(\bar \y) = -2 v_{\frac{(n-1)m}{2}} < 0$.   Hence $\A$ is not a strong Hankel tensor in this case.
\end{proof}

We consider the simple case that the tensors $\A$ are sixth order three dimensional quasi-truncated Hankel tensors.
We give a necessary condition that the sixth order three dimensional quasi-truncated Hankel tensors to be PSD, and a sufficient condition that the sixth order three dimensional quasi-truncated Hankel tensors to be SOS.

The sixth order three dimensional quasi-truncated Hankel tensor $\A$ is generated by $\vv = (v_0, v_1, 0, 0, 0, 0, v_6, 0, 0, 0, 0, v_{11}, v_{12})^\top \in \Re^{13}$. (\ref{e3-1}) and (\ref{e3-2}) have the simple form
\begin{equation} \label{e3-01}
f(\x) = v_0x_1^6 + 6v_1x_1^5x_2 + v_6(x_2^6 +30x_1x_2^4x_3 +90x_1^2x_2^2x_3^2 + 20x_1^3x_3^3) + 6v_{11}x_2x_3^5 + v_{12}x_3^6,
\end{equation}
and
\begin{equation} \label{e3-02}
g(\y) = v_0y_1^2 + 2v_1y_1y_2 + v_6(y_4^2 +2y_1y_7 +2y_2y_6 + 2y_3y_5)
+ 2v_{11}y_6y_7 + v_{12}y_7^2.
\end{equation}

To present a necessary condition for a sixth order three dimensional quasi-truncated Hankel tensor to be PSD, we first prove the following lemma.

\begin{Lemma} \label{l3-1}
Consider
$$\hat f(x_1, x_2) = v_0x_1^6 + 6v_1x_1^5x_2 + v_6x_2^6.$$
Then $\hat f$ is PSD if and only if $v_0 \ge 0$, $v_6 \ge 0$ and
\begin{equation} \label{e3-3}
|v_1| \le \left({v_0 \over 5}\right)^{5 \over 6}v_6^{1 \over 6}.
\end{equation}
\end{Lemma}
\begin{proof}  Suppose that $v_0 \ge 0$, $v_6 \ge 0$ and (\ref{e3-3}) holds.  Then, by the arithmetic-geometric inequality, one has
\begin{eqnarray*}
v_0x_1^6 + v_6x_2^6 & = & {1 \over 5}v_0x_1^6 + {1 \over 5}v_0x_1^6 + {1 \over 5}v_0x_1^6 + {1 \over 5}v_0x_1^6 + {1 \over 5}v_0x_1^6 + v_6x_2^6 \\
& \ge & 6\left(\left({v_0 \over 5}\right)^5x_1^{30}v_6x_2^6\right)^{1 \over 6} \\
& \ge & 6|v_1x_1^5x_2 |.
\end{eqnarray*}
This implies that $\hat f(x_1, x_2) \ge 0$ for any $(x_1, x_2)^\top \in \Re^2$, i.e., $\hat f(x_1, x_2)$ is PSD.

Suppose that $\hat f(x_1, x_2)$ is PSD.  It is easy to see that $v_0 \ge 0$ and $v_6 \ge 0$.   Assume now that (\ref{e3-3}) does not hold, i.e.,
\begin{equation} \label{e3-4}
|v_1| > \left({v_0 \over 5}\right)^{5 \over 6}v_6^{1 \over 6}.
\end{equation}
If $v_0 = v_6 = 0$, let $x_1 = 1$ and $x_2 = -v_1$.  Then $\hat f(x_1, x_2) < 0$.  We get a contradiction.   If $v_0=0$ and $v_6 \not = 0$, let
$x_1 = v_6^{1 \over 5}$ and $x_2 = -v_1^{1 \over 5}$.  Again, $\hat f(x_1, x_2) < 0$.  We get a contradiction.  Similarly, if $v_0 \not =0$ and $v_6 = 0$, we may get a contradiction.  If $v_0 \not =0$ and $v_6 \not = 0$, let
$x_1 = (5v_6)^{1 \over 6}$ and $x_2 = - {v_1 \over |v_1|}v_0^{1 \over 6}$.  Then by (\ref{e3-4}),
$$\hat f(x_1, x_2) = 6v_0v_6 - 6|v_1|(5v_6)^{5 \over 6}v_0^{1 \over 6} < 0.$$
We still get a contradiction.   This completes the proof.
\end{proof}

 We now present a necessary condition for a sixth order three dimensional quasi-truncated Hankel tensor to be PSD.

\begin{Proposition} \label{p3-2}
Suppose that (\ref{e1-3}) holds.   If $\A$ is a sixth order three dimensional PSD quasi-truncated Hankel tensor, then (\ref{e3-3}) and the following inequalities
\begin{equation} \label{e3-5}
|v_1| \le \left({v_{12} \over 5}\right)^{5 \over 6}v_6^{1 \over 6}
\end{equation}
and
\begin{equation} \label{e3-6}
\sqrt{v_0v_{12}} \ge 10 v_6
\end{equation}
hold.
If furthermore
\begin{equation} \label{e3-7}
v_1v_{12}^{5 \over 6} = v_{11}v_0^{5 \over 6},
\end{equation}
then (\ref{e2-8}) also holds.
\end{Proposition}
\begin{proof} Suppose that $\A$ is PSD.   In (\ref{e3-01}), let $x_3 = 0$.  By Lemma \ref{l3-1}, (\ref{e3-3}) holds.    In (\ref{e3-01}), let $x_1 = 0$.   By an argument similar to Lemma \ref{l3-1}, (\ref{e3-5}) holds.   In (\ref{e3-01}), let $x_2 = 0$.  Since $\A$ is PSD, we may easily get (\ref{e3-6}).

Suppose further that (\ref{e3-7}) holds. As in $[{\rm (i)} \Rightarrow {\rm (iii)}]$ part of the proof of Theorem \ref{t2-2}, we let $t \ge 0$ and let $\bar \x = (\bar x_1, \bar x_2, \bar x_3)^\top$, where
$\bar x_1 = v_{12}^{1 \over 6},\  \bar x_2 = \sqrt{t}(v_0v_{12})^{1 \over 12}, \
\bar x_3 = - v_0^{1 \over 6}.$
It follows from (\ref{e3-7}) that
\begin{equation}
6v_1\bar x_1^5\bar x_2 + 6 v_{11}\bar x_2\bar x_3^5 =0.
\end{equation}
This together with (\ref{e3-01}) implies that
\[
f(\bar \x)=v_0v_{12} + v_6(t^3 -30t^2 + 90t - 20)\sqrt{v_0v_{12}} + v_0v_{12} \ge 0.
\]
Proceed as in $[{\rm (i)} \Rightarrow {\rm (iii)}]$ part of the proof of Theorem \ref{t2-2}, we see that
 (\ref{e2-8}) holds in this case.
This completes the proof.
\end{proof}

We may also present a sufficient condition for a sixth order three dimensional quasi-truncated Hankel tensor to be SOS.
\begin{Proposition} \label{p3-4}
Let $\A$ be a sixth order three dimensional quasi-truncated Hankel tensor.
Suppose that  $v_0, v_6, v_{12} > 0$.   If there exist $t_1, t_2 > 0$ such that 
\begin{equation} \label{e3-8}
|v_1| \le {1 \over t_1} - {10v_6 \over t_1\sqrt{v_0v_{12}}}
\end{equation}
\begin{equation} \label{e3-9}
|v_{11}| \le {1 \over t_2} - {10v_6 \over t_2\sqrt{v_0v_{12}}}
\end{equation}
\begin{equation} \label{e3-10}
|v_1|\left({5 \over t_1v_0}\right)^5 +  |v_{11}|\left({5 \over t_2v_{12}}\right)^5 \le {\sqrt{70}-8 \over 2}v_6
\end{equation}
and
$$\left(v_0  - 10v_6\left({v_0 \over v_{12}}\right)^{1 \over 2} -|v_1|t_1v_0\right)\left(v_{12}  - 10v_6\left({v_{12} \over v_0}\right)^{1 \over 2} -|v_{11}|t_2v_{12}\right)$$
$$\times \left({\sqrt{70}-8 \over 2}v_6 -|v_1|\left({5 \over t_1v_0}\right)^5 - |v_{11}|\left({5 \over t_2v_{12}}\right)^5\right)$$
\begin{equation} \label{e3-11}
\ge {1 \over 27}v_6^3(60+15\sqrt{70})^3
\end{equation}
hold, then $\A$ is SOS.
\end{Proposition}
\begin{proof}  We write $f(\x) = \sum_{i=1}^5 f_i(\x)$, where
$$f_2(\x) = 10v_6 \left(\left({v_0 \over v_{12}}\right)^{1 \over 4}x_1^3 + \left({v_{12} \over v_0}\right)^{1 \over 4}x_3^3\right)^2,$$
$$f_3(\x) = |v_1|t_1v_0x_1^6+6v_1x_1^5x_2 + |v_1|\left({5 \over t_1v_0}\right)^5x_2^6,$$
$$f_4(\x) = |v_{11}|t_2v_{12}x_3^6+6v_{11}x_3^5x_2 + |v_{11}|\left({5 \over t_2v_{12}}\right)^5x_2^6,$$
$$f_5(\x) = v_6\left( \sqrt{10-\sqrt{70} \over 2}x_2^3 + \sqrt{150+15\sqrt{70}}x_1x_2x_3\right)^2$$
and
\begin{eqnarray*}
f_1(\x) & = & \left(v_0  - 10v_6\left({v_0 \over v_{12}}\right)^{1 \over 2} -|v_1|t_1v_0\right)x_1^6 + \left({\sqrt{70}-8 \over 2}v_6 -|v_1|\left({5 \over t_1v_0}\right)^5 -  |v_{11}|\left({5 \over t_2v_{12}}\right)^5\right)x_2^6 \\
&& - \left(v_{12}  - 10v_6\left({v_{12} \over v_0}\right)^{1 \over 2} -|v_{11}|t_2v_{12}\right)x_3^6 - v_6(60+15\sqrt{70})x_1^2x_2^2x_3^2.
\end{eqnarray*}
Clearly, $f_2$ and $f_5$ are squares.  From Lemma 1, we may show that $f_3$ and $f_4$ are PSD.   Since each of $f_3$ and $f_4$ has only two variables, they are SOS.   If (\ref{e3-8}-\ref{e3-11}) hold, by the arithmetic-geometric inequality, $f_1$ is PSD.  In this case, $f_1$ is a PSD diagonal minus tail form.   By \cite{FK}, $f_1$ is SOS.   Thus, if (\ref{e3-8}-\ref{e3-11}) hold, then $f$, hence $\A$, is SOS.
\end{proof}

\section{A Class of SOS Hankel Tensors}
\hspace{4mm}

In this section, we provide further classes for SOS Hanke tensors and examples for SOS Hankel tensors which are not strong Hankel tensors.

We say that $\A$ is a strong Hankel tensor generated by an absolutely integrable real valued function
$h:(-\infty, +\infty) \rightarrow [0, +\infty)$ if  it is a Hankel tensor and  its generating vector $\vv = (v_0, v_1, \cdots, v_{(n-1)m})^\top$ satisfies
\begin{equation}\label{e-str}
  v_k =\int_{-\infty}^{\infty}t^kh(t)dt, \quad k = 0, 1, \cdots ,(n-1)m.
\end{equation}
Such a real valued function $h$ is called the generating function of the strong Hankel tensor $\A$. It has been shown in \cite{Qi15} that any
strong Hankel tensor generated by an absolutely integrable real valued nonnegative function is a strong Hankel tensor.

We now define the completely decomposable tensor.

\begin{Definition}
 Let $\A \in  S_{m, n}$.  If there are positive integer $r \in \mathbb{N}$ and vectors $\x_j \in \Re^n$ for $j \in [r]$ such that
\begin{equation} 
\A = \sum_{j \in [r]} \x_j^{\otimes m},\end{equation}
then we say that $\A$ is a {\bf completely decomposable tensor}.

\end{Definition}

Then in the following theorem, we will show that when the order is even,
a strong Hankel tensor generated by an absolutely integrable real valued nonnegative function is indeed a limiting point of complete Hankel
tensors, which is a completely decomposable tensor.

\begin{Theorem}{(\bf Completely decomposability of strong Hankel tensors)}\label{th:2}
Let $m,n \in \mathbb{N}$. Let $\A$ be an $m$th-order
$n$-dimensional strong Hankel tensor generated by an absolutely integrable real valued nonnegative function. If the order $m$ is an even
number, then $\mathcal{A}$ is a completely decomposable tensor and a
limiting point of  complete Hankel tensors. If the order $m$ is an
odd number, then $\mathcal{A}$ is
a completely $r$-decomposable tensor with
$r=(n-1)m+1$.
\end{Theorem}

\begin{proof}

Let $h$ be the generating function of the strong Hankel tensor $\mathcal{A}$. Then, for any $\x\in \Re^n$,
\begin{eqnarray}\label{eq:1}
f(\x):=\A \x^{\otimes m} & = &  \sum_{i_1,i_2,\ldots,i_m=1}^n v_{i_1+i_2+\ldots+i_m-m} x_{i_1}x_{i_2} \ldots x_{i_m} \nonumber \\
& = & \sum_{i_1,i_2,\ldots,i_m=1}^n  \left(\int_{-\infty}^{+\infty}  t^{i_1+i_2+\ldots+i_m-m} h(t)   dt \right)\, x_{i_1}x_{i_2} \ldots x_{i_m} \nonumber \\
&=&\int_{-\infty}^{+\infty}  \left(\sum_{i_1,i_2,\ldots,i_m=1}^n t^{i_1+i_2+\ldots+i_m-m} x_{i_1}x_{i_2} \ldots x_{i_m}  \right)h(t)dt \, \nonumber \\
&=& \int_{-\infty}^{+\infty}  \left(\sum_{i=1}^n t^{i-1} x_{i}\right)^m h(t)dt = \lim_{l \rightarrow +\infty} f_l(x),
\end{eqnarray}
where
\[
f_l(\x)=\int_{-l}^{l}  \left(\sum_{i=1}^n t^{i-1} x_{i}\right)^m h(t)dt.
\]
By the definition of Riemann integral, for each $l \ge 0$, we have
$f_l(\x)  = \lim_{k\rightarrow \infty} f_l^k(\x)$,
where $f_l^k(\x)$ is a polynomial defined by
\[
f_l^k(\x):=\sum_{j=0}^{2kl} \frac{\left(\sum_{i=1}^n
(\frac{j}{k}-l)^{i-1} x_i \right)^m h(\frac{j}{k}-l)}{k}.
\]
Fix any $l \ge 0$ and $k \in \mathbb{N}$.
Note that
\begin{eqnarray*}
f_l^k(\x):=\sum_{j=0}^{2kl} \frac{\left(\sum_{i=1}^n
(\frac{j}{k}-l)^{i-1} x_i \right)^m h(\frac{j}{k}-l)}{k}&= & \sum_{j=0}^{2kl}\left(\sum_{i=1}^n\frac{
(\frac{j}{k}-l)^{i-1}  h(\frac{j}{k}-l)^{\frac{1}{m} }}{k^{\frac{1}{m}}} x_i\right)^m \\
& = & \sum_{j=0}^{2kl} (\langle \uu_j, \x \rangle)^m,
\end{eqnarray*}
where $\uu_j=\frac{
 h\left(\frac{j}{k}-l\right)^{\frac{1}{m} }}{k^{\frac{1}{m}}}\left(1,\frac{j}{k}-l,\ldots,(\frac{j}{k}-l)^{n-1}\right)$. Here $\uu_j$ are always well-defined as
$h$ takes nonnegative values \cite{Qi15}. Define $\A_l^k$ be a
symmetric tensor such that
$f_l^k(\x)=\A_l^k \x^{\otimes m}$.
Then, it is easy to see that
 each $\A_k^l$ is a complete Hankel tensor and thus a completely decomposable tensor. {  Note that} the completely decomposable tensor cone $CD_{m,n}$ is a closed convex cone when $m$ is even.
 It then follows that
$\A =\lim_{k \rightarrow \infty}\lim_{l \rightarrow
\infty}\A_k^l$ is a  {  completely decomposable tensor and a limiting point of complete Hankel tensors.}

To see the assertion in the odd order case, we use a similar
argument as in \cite{Qi15}. Pick real numbers
$\gamma_1,\ldots,\gamma_{r}$ with $r=(n-1)m+1$ and $\gamma_i \neq
\gamma_j$ for $i \neq j$. Consider the following linear equation in
$\alpha=(\alpha_1,\ldots,\alpha_r)$ with
\[
v_k=\sum_{i=1}^r \alpha_i \gamma_i^{k}, \ k=0,\ldots,(n-1)m.
\]
Note that this linear equation always has a solution say $\bar \alpha= (\bar \alpha_1,\ldots,\bar \alpha_r)$ because the matrix in the above linear equation is a
nonsingular Vandermonde matrix. Then, we see that
\[
\A_{i_1,\ldots,i_m} =v_{i_1+\ldots+i_m-m}=\sum_{i=1}^r \bar
\alpha_i \gamma_i^{i_1+\ldots+i_m-m}=\sum_{i=1}^r \bar \alpha_i
\left((\uu_i)^{\otimes m}\right)_{i_1,\ldots,i_m},
\]
where $\uu_i \in \Re^n$ is given by
$\uu_i=(1,\gamma_i,\ldots,\gamma_i^{n-1})^T$. This shows that
$\A=\sum_{i\in [r]} \bar \alpha_i (\uu_i)^{\otimes m}$.
Now, as $m$ is an odd
number, we have
$\A =\sum_{i=1}^r \left({\bar \alpha_i}^{\frac{1}{m}}
\uu_i\right)^{\otimes m}$.
Therefore, $\mathcal{A}$ is a completely decomposable tensor and the last conclusion follows.
\end{proof}

In \cite{Qi15}, the author has provided an example of  positive semi-definite Hankel tensors with order $m=4$, which is not a strong Hankel tensor generated by an absolutely integrable real valued nonnegative function.  We now extend this example to the general case where $m = 2k$ for any integer $k \ge 2$.  We will also further show that such tensors are indeed SOS tensors but not completely decomposable {  (and so, are  not
strong Hankel tensors generated by an absolutely integrable real valued nonnegative function by Theorem \ref{th:2})}.

Let $m = 2k$, $n$ = 2, $k$ is an integer and $k \ge 2$.   {  Let $v_0 = v_m = 1$, $v_{2l} = v_{m-2l} = -{1 \over \left({m \atop 2l}\right)}$, $l=1,\ldots,k-1$, and $v_j = 0$ for other $j$. Let $\A = (a_{i_1\cdots i_m})$ be defined by
$a_{i_1\cdots i_m} = v_{i_1+\cdots +i_m-m},$
for $i_1, \cdots, i_m = 1, 2$.}   Then $\A$ is an even order Hankel tensor.
{  For any $\x \in \Re^2$, we have
$$\A \x^{\otimes m} = x_1^m - \sum_{j=1}^{k-1} x_1^{m-2j} x_2^{2j}   + x_2^m = \sum_{j=0}^{k-2}\left(x_1^{k-j}x_2^j - x_1^{k-j-2}x_2^{j+2}\right)^2.$$
Thus,} $\A$ is an SOS-Hankel tensor, hence a positive semi-definite Hankel tensor.  On the other hand, $\A$ is not a completely decomposable tensor.  Assume that $\A$ is a completely decomposable tensor.   Then there are vectors $\uu_j = (a_j, b_j)^\top$ for $j \in [r]$ such that
$\A = \sum_{j=1}^r \uu_j^m.$
Then for any $\x \in \Re^2$,
$$\A \x^{\otimes m} = \sum_{p=1}^r (a_px_1+b_px_2)^m = \sum_{j=0}^m \sum_{p=1}^r \left({ m \atop j}\right)a_p^{m-j}b_p^jx_1^{m-j}x_2^j.$$
{  On the other hand,
$$\A \x^{\otimes m} = x_1^m -  \sum_{j=1}^{k-1} x_1^{m-2j} x_2^{2j} + x_2^m.$$}
Comparing the coefficients of $x_1^{m-2}x_2^2$ in the above two expressions of $\A \x^m$, we have
$$\sum_{p=1}^r \left({ m \atop 2}\right)a_p^{m-2}b_p^2 = -1.$$
This is impossible.   Thus, $\A$ is not completely decomposable (and so, is not a strong Hankel tensor generated by an absolutely integrable real valued nonnegative function).

\bigskip

We now have two further questions:

1. Is a completely decomposable Hankel tensor always a strong Hankel tensor?

2. Is a truncated Hankel tensor completely decomposable?


\end{document}